\documentclass[11pt,reqno]{amsart}

\usepackage{amssymb,amsmath,amsthm,amscd,latexsym,amsfonts}
\usepackage{mathtools}
\usepackage[T1]{fontenc}

\usepackage{graphicx}
\usepackage{color}
\usepackage[a4paper,top=3cm, bottom=3cm, left=3cm, right=2cm]{geometry}

\newtheorem{thm}{Theorem}

\newtheorem{lemma}{Lemma}
\newtheorem{pro}{Proposition}

\newtheorem{cor}{Corollary}

\numberwithin{equation}{section} 

\newcommand{\bea}{\begin{eqnarray}}
\newcommand{\eea}{\end{eqnarray}}

\newcommand{\R}{\mathbb{R}}

\def\R{\mathbb{R}}

\begin{document}
\title [Dynamics of the China's five element philosophy]
{Dynamics of quadratic operators generated by China's Five elements philosophy}

\author {N.N.Ganikhodjaev, C.H. Pah, U.A. Rozikov}

\address{N.\ N.Ganikhodjaev\\ Department of Computational and Theoretical Sciences, Faculty of Science, IIUM, Kuantan, Malaysia.}
 \email {gnasir@iium.edu.my}

 \address{C.\ H. Pah\\ Department of Computational and Theoretical Sciences, Faculty of Science, IIUM, Kuantan, Malaysia.}
 \email {pahchinhee@iium.edu.my}

 \address{U.\ A.\ Rozikov\\ Institute of mathematics, 29, Do'rmon Yo'li str., 100125,
Tashkent, Uzbekistan.} \email {rozikovu@yandex.ru}

\begin{abstract} Motivating by the China's five element philosophy (CFEP)
we construct a permuted Volterra quadratic
stochastic operator acting on the four dimensional simplex.
This operator (depending on 10 parameters) is considered
as an evolution operator for CFEP.
We study the discrete time dynamical system generated by this operator.
Mainly our results related to a symmetric operator (depending on one parameter).
We show that this operator has a unique fixed point, which is repeller.
Moreover, in the case of non-zero parameter, it has two 5-periodic orbits.
We divide the simplex to four subsets: the first set consists a single point (the fixed point);
the second (resp. third) set is the set of initial points trajectories of which converge
to the first (resp. second) 5-periodic orbit; the fourth subset is the set of initial
points trajectories of which do not converge
and their sets of limit points are infinite and lie on the boundary of the simplex.
We give interpretations of our results to CFEP. 
\end{abstract}



\maketitle

\section{Introduction}

The China's Five element philosophy model (CFEPM) gives interaction between 5 phases: Wood, Fire, Earth, Metal and Water.
They are usually used to describe the state in nature, five phases are around 72 days each:\footnote{https://en.wikipedia.org/wiki/Wu$_-$Xing}

\emph{Wood/Spring:} a period of growth, which generates abundant wood and vitality,

\emph{Fire/Summer:} a period of swelling, flowering, brimming with fire and energy,

\emph{Earth:} the in-between transitional seasonal periods, or a separate 'season'
known as Late Summer or Long Summer - in the latter case
associated with leveling and dampening (moderation) and fruition,

\emph{Metal/Autumn:} a period of harvesting and collecting,

\emph{Water/Winter:} a period of retreat, where stillness and storage pervades.

The doctrine of five phases describes two cycles, a generating or creation cycle,
and an overcoming or destruction cycle of interactions between the phases (see Fig. \ref{ff1}).
  \begin{figure}
\centering
\includegraphics[width=8cm]{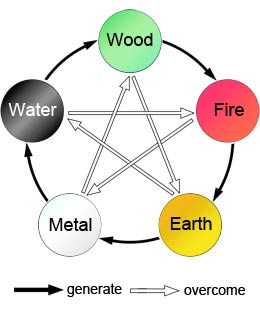}
\caption{Five elements (phases) and the interactions between them. Source: https://www.travelchinaguide.com/intro/astrology/five-elements.htm }
\label{ff1}
\end{figure}

For simplicity of notations we denote
$$1=Wood, \ \ 2=Fire, \ \ 3=Earth, \ \ 4=Metal, \ \ 5=Water.$$

Let $x=(x_1, x_2, x_3, x_4, x_5)$ be a probability distribution on the set $E=\{1,2,3,4,5\}$.
Consider this vector as a state of the CFEPM, and let an operator $W$ describes its future evolution, i.e.,
if in the next 'generation' the state of the CFEPM is $x'$ then $x'=W(x)$. This operator is called an evolution
operator.

In this paper using the interactions between phases we construct
an evolution operator $W$ acting on the four dimensional simplex.
This operator is quadratic stochastic operator (QSO) (see \cite{K}, \cite{L}),
which is non-Volterra, but permuted Volterra QSO.

In \cite{G}, \cite{G1}, using theory of the Lyapunov
function and tournaments, the theory of Volterra QSO was developed.
But non-Volterra QSOs were not in completely studied. Because there is no any general
theory which can be applied for investigation of non-Volterra operators.
In recent years a non-Volterra QSO has become of interest in connection with
its numerous applications in many branches of mathematics, biology and physics
 (see \cite{G2}-\cite{GMR}, \cite{GM}-\cite{GD}, \cite{M}, \cite{ME}, \cite{RS}-\cite{RJ} for the theory and applications of such operators).
Since there is no any general
theory for investigation of non-Volterra operators, each such operator requires a corresponding approach.

The paper is organized as follows:
In Sec. 2, we give some preliminary definitions. In Section 3 we construct an evolution operator describing CFEPM.
Section 4 is devoted to a detailed analysis of a symmetric case. There we give all fixed points, periodic points
and study behavior of all trajectories. It is shown that there is not converging trajectory:
some of them go to a 5-periodic orbit, but others have an infinite set of limit points being a subset of the boundary of the simplex.
We give interpretations of our results to CFEP in Section 5. 
\section{Definitions}\label{S:definitions}

Let $E=\{1,\dots,m\}$ be a finite set. The set of all probability measures on $E$ is
\[ S^{m-1}=\{x=(x_1,\dots,x_m)\in \R^m: x_i\geq 0, \ \sum^m_{i=1}x_i=1 \} \]
and called $(m-1)$-dimensional {\it simplex}.

The boundary of the simplex is
$$\partial S^{m-1}=\{x\in S^{m-1}: \prod_{i=1}^mx_i=0\}.$$

The inside of the simplex is
$${\rm int} S^{m-1}=\{x\in S^{m-1}: \prod_{i=1}^mx_i>0\}.$$

A \emph{quadratic stochastic operator} (QSO) is a
(quadratic) mapping from the simplex
into itself, of the form
\begin{equation}\label{ov}
 V: x_k'=\sum^m_{i,j=1}P_{ij,k}x_ix_j, \qquad (k=1,\dots,m),
\end{equation}
where $P_{ij,k}$ are the coefficients of heredity and
\begin{equation}\label{as1}
P_{ij,k}\geq 0, \qquad \sum^m_{k=1}P_{ij,k}=1, \qquad  (i,j,k=1,\dots,m).
\end{equation}
A matrix $(P_{ij,k})$ satisfying (\ref{as1}) is called stochastic.

 The discrete-time dynamical
 system generated by $V$ is the following sequence
 \begin{equation}\label{2}
 x^{(0)},\ \ x^{(1)}= V(x^{(0)}), \ \ x^{(2)}=V^{2}(x^{(0)}),\ \  x^{(3)}= V^{3}(x^{(0)}),\dots
\end{equation}
where $V^n(x)=\underbrace{V(V(...V}_n(x))...)$ denotes the $n$ times
  iteration of $V$ to $x$.

The sequence (\ref{2}) is also called a {\it trajectory} (or forward {\it orbit}) of
$x^{(0)}$ under action of the operator $V$.  For a review on dynamical systems see \cite{Rb}.

{\bf The main problem} for a given dynamical system is to
describe the limit points of $\{x^{(n)}\}_{n=0}^\infty$ for
arbitrary given $x^{(0)}$ (see \cite{GMR} for a review).

Let $\{x^{(n)}\}_{n=1}^\infty$ be the trajectory of the point
$x^{(0)}\in S^{m-1}$ under QSO (\ref{ov}). Denote by $\omega(x^0)$ the set of
limit points of the trajectory. Since $S^{m-1}$ is compact and  $\{x^{(n)}\}\subset S^{m-1}$,
it follows that $\omega(x^0)\ne
\emptyset.$
If $\omega(x^0)$ consists of a single point,
then the trajectory converges, and $\omega(x^0)$ is a fixed point of
(\ref{ov}).

A {\it Volterra QSO} is defined by the additional assumption
$$ P_{ij,k}=0, \ \ \mbox{if}\ \ k\not\in \{i,j\},\,\,
\forall i,j,k\in E.$$

The general form of Volterra QSO  (see \cite{G})
$$ V: x= (x_{1},...,x_{m}) \in S^{m-1}\,
\rightarrow\,V(x)= x'= (x'_1,...,x'_m)\,\in S^{m-1}
$$ is given by
$$x'_k=x_k\left(1+\sum^m_{i=1}a_{ki}x_i\right),\ \ k\in E,$$
where
$$a_{ki}=2P_{ik,k}-1 \ \ \mbox{for}\, i\neq k\,\,\mbox{and}\
\,a_{ii}=0, \ i\in E.$$  Moreover
$$a_{ki}=-a_{ik}\ \ \mbox{and} \ \ |a_{ki}| \leq 1.$$
Denote by $A=(a_{ij})_{i,j=1}^m$ the skew-symmetric matrix with
entries defined above.

\section{Evolution operator}
Denote by $P_{ij,k}$ the probability of the event that $i$ and $j$ interact and product $k$.
Assume that
\begin{equation}
P_{ij,k}=P_{ji,k}, \forall i,j,k\in \{1,2,3,4,5\}.
\end{equation}
By this condition it suffices to define $P_{ij,k}$ for any $i\leq j$.

The Fig. \ref{ff1} suggests to consider the following
probabilities:

Since the generating interactions of the five elements are like
the conception we can assume that
$$P_{11,2}=P_{22,3}=P_{33,4}=P_{44,5}=P_{55,1}=1,$$
that is wood fuels fire, fire forms earth, earth contains metal, metal carries water and
water feeds wood.
Considering the generating and overcoming interactions we have
$$\begin{array}{llllllllll}
P_{12,1}=0, \ \  P_{12,2}=a, \ \ P_{12,3}=1-a, \ \ P_{12,4}=0, \ \ P_{12,5}=0\\[2mm]
P_{13,1}=0, \ \  P_{13,2}=b, \ \ P_{13,3}=0, \ \ P_{13,4}=1-b, \ \ P_{13,5}=0\\[2mm]
P_{14,1}=0, \ \  P_{14,2}=c, \ \ P_{14,3}=0, \ \ P_{14,4}=0, \ \ P_{14,5}=1-c\\[2mm]
P_{15,1}=d, \ \  P_{15,2}=1-d, \ \ P_{15,3}=0, \ \ P_{15,4}=0, \ \ P_{15,5}=0\\[2mm]
P_{23,1}=0, \ \ P_{23,2}=0, \ \ P_{23,3}=\alpha, \ \ P_{23,4}=1-\alpha, \ \ P_{23,5}=0\\[2mm]
P_{24,1}=0, \ \ P_{24,2}=0, \ \ P_{24,3}=\beta, \ \ P_{24,4}=0, \ \ P_{24,5}=1-\beta\\[2mm]
P_{25,1}=\gamma, \ \ P_{25,2}=0, \ \ P_{25,3}=1-\gamma, \ \ P_{25,4}=0, \ \ P_{25,5}=0\\[2mm]
P_{34,1}=0, \ \ P_{34,2}=0, \ \ P_{34,3}=0, \ \ P_{34,4}=\delta, \ \ P_{34,5}=1-\delta\\[2mm]
P_{35,1}=\theta, \ \ P_{35,2}=0, \ \ P_{35,3}=0, \ \ P_{35,4}=1-\theta, \ \ P_{35,5}=0\\[2mm]
P_{45,1}=\omega, \ \ P_{45,2}=0, \ \ P_{45,3}=0, \ \ P_{45,4}=0, \ \ P_{45,5}=1-\omega
\end{array}
$$
where $a,b,c,d,\alpha, \beta, \gamma, \delta, \theta, \omega\in [0,1]$.

These probabilities define the following quadratic stochastic operator $W:S^4\to S^4$:
\begin{equation}\label{w}
W:\left\{\begin{array}{lllll}
x_1'=x_5\left[1+(2d-1)x_1+(2\gamma-1)x_2+(2\theta-1)x_3+(2\omega-1)x_4\right]\\[2mm]
x_2'=x_1\left[1+(2a-1)x_2+(2b-1)x_3+(2c-1)x_4-(2d-1)x_5\right]\\[2mm]
x_3'=x_2\left[1-(2a-1)x_1+(2\alpha-1)x_3+(2\beta-1)x_4-(2\gamma-1)x_5\right]\\[2mm]
x_4'=x_3\left[1-(2b-1)x_1-(2\alpha-1)x_2+(2\delta-1)x_4-(2\theta-1)x_5\right]\\[2mm]
x_5'=x_4\left[1-(2c-1)x_1-(2\beta-1)x_2-(2\delta-1)x_3-(2\omega-1)x_5\right]
\end{array}\right.
\end{equation}
Consider permutation
$$\pi =\left(\begin{array}{ll}
1 \ \ 2 \ \ 3 \ \ 4 \ \ 5\\[2mm]
5 \ \ 1 \ \ 2 \ \ 3 \ \ 4
\end{array}\right)$$
Define mapping $T_\pi :S^4\to S^4$ as
$$T_\pi(x)=T_\pi((x_1,\dots,x_5))=(x_5, x_1, x_2, x_3, x_4).$$
Consider also the following Volterra operator $V: S^4\to S^4$ defined by
\begin{equation}
V:\left\{\begin{array}{lllll}
y_1=x_1\left[1+(2a-1)x_2+(2b-1)x_3+(2c-1)x_4-(2d-1)x_5\right]\\[2mm]
y_2=x_2\left[1-(2a-1)x_1+(2\alpha-1)x_3+(2\beta-1)x_4-(2\gamma-1)x_5\right]\\[2mm]
y_3=x_3\left[1-(2b-1)x_1-(2\alpha-1)x_2+(2\delta-1)x_4-(2\theta-1)x_5\right]\\[2mm]
y_4=x_4\left[1-(2c-1)x_1-(2\beta-1)x_2-(2\delta-1)x_3-(2\omega-1)x_5\right]\\[2mm]
y_5=x_5\left[1+(2d-1)x_1+(2\gamma-1)x_2+(2\theta-1)x_3+(2\omega-1)x_4\right]
\end{array}\right.
\end{equation}

Note that $V$ is non-ergodic Volterra QSO \cite{GGJ}.

The corresponding skew-symmetric matrix of $V$ is the following
 \begin{equation}
A=\left(\begin{array}{ccccc}
0& 2a-1&2b-1&2c-1&-(2d-1)\\[2mm]
-(2a-1)&0&2\alpha-1&2\beta-1&-(2\gamma-1)\\[2mm]
-(2b-1)&-(2\alpha-1)&0&2\delta-1&-(2\theta-1)\\[2mm]
-(2c-1)&-(2\beta-1)&-(2\delta-1)&0&-(2\omega-1)\\[2mm]
2d-1&2\gamma-1&2\theta-1&2\omega-1&0
\end{array}\right).
\end{equation}
The following lemma is obvious
\begin{lemma} We have
$$W=T_\pi\circ V.$$
\end{lemma}
Thus the evolution operator $W$, (\ref{w}), is a permuted Volterra operator.
As it was shown in \cite{RJ} the permuted Volterra operators may have completely
different dynamical systems compared with Volterra ones. In particular, the operator may have
periodic orbits (but Volterra operator does not have periodic orbits \cite{G}).

Denote $e_i=(\delta_{1i}, \delta_{2i}, \delta_{3i}, \delta_{4i}, \delta_{5i})$, $i=1,2,3,4,5$,
where $\delta_{ij}$ is the Kronecker symbol.

 The following 5-periodic orbit of $W$ is easily checked:
 $$e_1\to e_2\to e_3\to e_4\to e_5\to e_1.$$
For simplicity we introduce
$$A=2a-1, \ \ B=2b-1, \ \ C=2c-1, \ \ D=2d-1, \ \ E=2\alpha-1,$$
$$F=2\beta-1, \ \ G=2\gamma-1, \ \ H=2\delta-1, \ \ I=2\theta-1, \ \ J=2\omega-1.$$
We have $A,B,C,D,E,F,G,H,I,J\in [-1,1]$.

A fixed point of the operator $W$ is a solution of the equation $W(x)=x$, i.e., the following system
\begin{equation}\label{wfp}
\left\{\begin{array}{lllll}
x_1=x_5\left[1+Dx_1+Gx_2+Ix_3+Jx_4\right]\\[2mm]
x_2=x_1\left[1+Ax_2+Bx_3+Cx_4-Dx_5\right]\\[2mm]
x_3=x_2\left[1-Ax_1+Ex_3+Fx_4-Gx_5\right]\\[2mm]
x_4=x_3\left[1-Bx_1-Ex_2+Hx_4-Ix_5\right]\\[2mm]
x_5=x_4\left[1-Cx_1-Fx_2-Hx_3-Jx_5\right]
\end{array}\right.
\end{equation}
\begin{lemma}\label{lf1} There is no fixed point of $W$ on the boundary $\partial S^4$.
\end{lemma}
\begin{proof} Assume there is a fixed point $x^*=(x^*_1, x^*_2, x^*_3, x^*_4, x^*_5)\in \partial S^4$, then there is $i\in \{1,2,3,4,5\}$
such that $x_i^*=0$. Using this zero-coordinate from (\ref{wfp}) we get $x^*_1=x^*_2=x^*_3=x^*_4=x^*_5=0$ this contradiction
to that $x^*\in S^4$.
\end{proof}

In general the systems (\ref{wfp}) seems difficult to solve.
Therefore we consider some special choice of parameters.
\section{A symmetric case}
Assume
$$A=p, B=-p, C=p, D=p, E=p, F=-p, G=-p, H=p, I=p, J=-p,$$
where $p\in[-1,1]$. In this case the operator is

\begin{equation}\label{w1}
W: \left\{\begin{array}{lllll}
x'_1=x_5\left[1+p(x_1-x_2+x_3-x_4)\right]\\[2mm]
x'_2=x_1\left[1+p(x_2-x_3+x_4-x_5)\right]\\[2mm]
x'_3=x_2\left[1-p(x_1-x_3+x_4-x_5)\right]\\[2mm]
x'_4=x_3\left[1+p(x_1-x_2+x_4-x_5)\right]\\[2mm]
x'_5=x_4\left[1-p(x_1-x_2+x_3-x_5)\right]
\end{array}\right.
\end{equation}
Permuting coordinates of (\ref{w1}) we get a Volterra operator (as before denoted by $V$).

The following lemma is very useful
\begin{lemma}\label{lT} For (\ref{w1}) we have
\begin{itemize}
\item[1.] $W=T_\pi\circ V=V\circ T_\pi.$
\item[2.] $T_\pi\circ V^n=V^n\circ T_\pi$
\item[3.] $W^n=T_\pi^i \circ V^n \ \ \mbox{if} \ \ n=5k+i, \ \ i=0,1,2,3,4, \ \ k=0,1,2,\dots$
\end{itemize}
\end{lemma}
\begin{proof}
1. Straightforward.

2. Obtained by iterating the equality of part 1.

3. Since $T^5_\pi=id$ using parts 1 and 2 we get $W^{5k}=V^{5k}$, indeed,
$$W^5=W^4(T_\pi\circ V)=W^3(T_\pi\circ V\circ T_\pi\circ V)=W^3(T^2_\pi\circ V^2)=W^2(T_\pi\circ V\circ T^2_\pi\circ V^2)$$
$$=W^2(T^3_\pi\circ V^3)=W(T_\pi\circ V\circ T^3_\pi\circ V^3)=W(T^4_\pi\circ V^4)=T^5_\pi\circ V^5=V^5.$$
This equality and $W^i=T^i_\pi \circ V^i$, $i=0,1,2,3,4$ gives
$$W^n=W^{5k+i}=W^{5k}(W^i)=V^{5k}(T^i_\pi \circ V^i)=T_\pi^i \circ V^n.$$

\end{proof}
\subsection{Fixed points}
The following proposition gives fixed points of (\ref{w1}):

\begin{pro} For any $p\in [-1,1]$ the operator (\ref{w1}) has a unique fixed point $$P=(1/5, 1/5, 1/5, 1/5, 1/5).$$
\end{pro}
\begin{proof}
The system of equations for the fixed point is
\begin{equation}\label{wf1}
\left\{\begin{array}{lllll}
x_1=x_5\left[1+p(x_1-x_2+x_3-x_4)\right]\\[2mm]
x_2=x_1\left[1+p(x_2-x_3+x_4-x_5)\right]\\[2mm]
x_3=x_2\left[1-p(x_1-x_3+x_4-x_5)\right]\\[2mm]
x_4=x_3\left[1+p(x_1-x_2+x_4-x_5)\right]\\[2mm]
x_5=x_4\left[1-p(x_1-x_2+x_3-x_5)\right]
\end{array}\right.
\end{equation}
If $x\in S^4$ is a fixed point then by Lemma \ref{lf1} we have $x_1x_2x_3x_4x_5>0$.
Then from (\ref{wf1}) we get
$$\psi(x)=\left[1+p(x_1-x_2+x_3-x_4)\right]\left[1+p(x_2-x_3+x_4-x_5)\right]\left[1-p(x_1-x_3+x_4-x_5)\right]$$
\begin{equation}\label{ch}
\times\left[1+p(x_1-x_2+x_4-x_5)\right]\left[1-p(x_1-x_2+x_3-x_5)\right]=1.
\end{equation}
For the LHS of (\ref{ch}) using the arithmetic-geometric mean inequality,
we get
$$\psi(x)=\left[1+p(x_1-x_2+x_3-x_4)\right]\left[1+p(x_2-x_3+x_4-x_5)\right]\left[1-p(x_1-x_3+x_4-x_5)\right]$$
\begin{equation}\label{ch1}
\times\left[1+p(x_1-x_2+x_4-x_5)\right]\left[1-p(x_1-x_2+x_3-x_5)\right]\leq 1
\end{equation}
and it is known that the equality in (\ref{ch1}) holds only if
$$\left[1+p(x_1-x_2+x_3-x_4)\right]=\left[1+p(x_2-x_3+x_4-x_5)\right]=\left[1-p(x_1-x_3+x_4-x_5)\right]$$
\begin{equation}\label{ch2}
=\left[1+p(x_1-x_2+x_4-x_5)\right]=\left[1-p(x_1-x_2+x_3-x_5)\right].
\end{equation}
Then from (\ref{ch}) it follows that each these numbers equal to 1 (i.e.,  $[...]=1$).
Consequently, from the system of equations (\ref{wf1}) we get $x_1=x_2=x_3=x_4=x_5$, and since $x\in S^4$
we obtain a unique fixed point $P=(1/5, 1/5, 1/5, 1/5, 1/5)$.
\end{proof}
\begin{pro}\label{pb} The fixed point, $P$, is a repeller for any $p\ne 0$ and it is a saddle point for $p=0$.
\end{pro}
\begin{proof}
The proof is based on the analysis of eigenvalues of the Jacobian
at the fixed point $P$. Denoting $q=p/5$, the equation for eigenvalues can be written as
\begin{equation}\label{e4}\lambda^4+(-5q+1)\lambda^3+(15q^2+1)\lambda^2+(-15q^3-5q^2-5q+1)\lambda+5q^4+10q^2+1=0.
\end{equation}
Note that the LHS of this equation is positive for any $q\in [-1/5, 1/5]$, to see this it suffices
to show that $(15q^2+1)\lambda^2+(-15q^3-5q^2-5q+1)\lambda+5q^4+10q^2+1>0$. The last inequality
can be checked by showing that the discriminant of the quadratic equation is negative:
$$D(q)=(-15q^3-5q^2-5q+1)^2-4(15q^2+1)(5q^4+10q^2+1)<0.$$
Since $q$ is in finite set, the last inequality can be seen from the graph of function $D(q)$ (see Fig. \ref{fD}).
  \begin{figure}
\centering
\includegraphics[width=10cm]{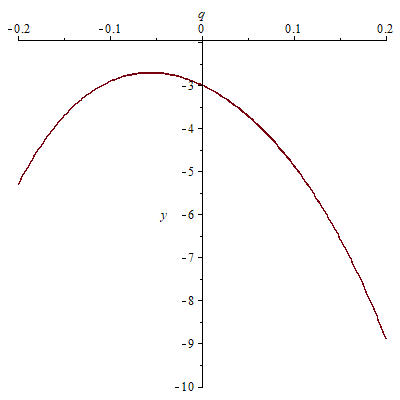}
\caption{ The graph of $D(q)$ in the domain $[-0.2, 0.2]$ (where it is defined).}
\label{fD}
\end{figure}
Thus LHS of (\ref{e4}) is strictly positive, therefore it has four {\it complex} solutions.
By a computer or using known formulas\footnote{see https://en.wikipedia.org/wiki/Quartic$_-$function}
one can obtain an explicit solution of the
quartic equation (as a function of $q=p/5\in [-0.2,0.2]$),
then the absolute values of two non-conjugate solutions has the form
$$f_1(q)=(5-2\sqrt{5})q^2+1,$$
$$f_2(q)=(5+2\sqrt{5}) q^2+1.$$
One can see that they are >1 iff $q\ne 0$.
\end{proof}

\subsection{Periodic points}

For $p=0$ operator (\ref{w1}) coincides with $T_\pi$, for which any point is 5-periodic.

Let us first solve the equation
\begin{equation}\label{pe1}
W(x)=T_\pi (x)=(x_5, x_1, x_2, x_3, x_4).
\end{equation}
This is equivalent (for $p\ne 0$) to
\begin{equation}\label{pe2}
\left\{\begin{array}{lllll}
x_5(x_1-x_2+x_3-x_4)=0\\[2mm]
x_1(x_2-x_3+x_4-x_5)=0\\[2mm]
x_2(x_1-x_3+x_4-x_5)=0\\[2mm]
x_3(x_1-x_2+x_4-x_5)=0\\[2mm]
x_4(x_1-x_2+x_3-x_5)=0
\end{array}\right.
\end{equation}
The following lemma says that solutions of this system are invariant
with respect to permutation operator $T_\pi$.
\begin{lemma}\label{k} If $x=(x_1, x_2, x_3, x_4, x_5)$ is a solution to (\ref{pe2})
then $T_\pi(x)$ is also its solution.
\end{lemma}
\begin{proof} Putting in (\ref{pe2}) the vector $(x_5, x_1, x_2, x_3, x_4)$ we get
\begin{equation}\label{pe3}
\left\{\begin{array}{lllll}
x_4(x_5-x_1+x_2-x_3)=0\\[2mm]
x_5(x_1-x_2+x_3-x_4)=0\\[2mm]
x_1(x_5-x_2+x_3-x_4)=0\\[2mm]
x_2(x_5-x_1+x_3-x_4)=0\\[2mm]
x_3(x_5-x_1+x_2-x_4)=0
\end{array}\right.
\end{equation}
Comparing the first (resp. 2-nd, 3-rd, 4-th, 5-th) equation of (\ref{pe3}) with the 5-th (resp. 1-st, 2-nd, 3-rd, 4-th)
equation of (\ref{pe2}) one can see that they are identical.
\end{proof}
 As corollary of this lemma we get
 \begin{pro} If $x$ is a solution to (\ref{pe1}) then
$W^i(x)=T_\pi^i(x), \ \ i=1,2,3,4,5$.
\end{pro}
 \begin{proof} We have
\begin{equation}\label{pe4}
W^2(x)=W(W(x))=W(T_\pi(x))=\left\{\begin{array}{lllll}
x_4[1+p(x_5-x_1+x_2-x_3)]\\[2mm]
x_5[1+p(x_1-x_2+x_3-x_4)]\\[2mm]
x_1[1-p(x_5-x_2+x_3-x_4)]\\[2mm]
x_2[1+p(x_5-x_1+x_3-x_4)]\\[2mm]
x_3[1-p(x_5-x_1+x_2-x_4)]
\end{array}\right.
\end{equation}
From (\ref{pe4}) by (\ref{pe3}) we get $W^2(x)=(x_4, x_5, x_1, x_2, x_3)=T^2_\pi(x).$

Using the last equality and Lemma \ref{k}, for $i=3$ (and then for $i=4,5$) we get
$$W^i(x)=W(W^{i-1}(x))=W(T^{i-1}_\pi(x))=T_\pi(T_\pi^{i-1}(x))=T^i_\pi(x).$$
 \end{proof}
 Thus we have the following
 \begin{cor} Any solution of (\ref{pe2}) generates a 5-periodic orbit for the operator (\ref{w1}).
  \end{cor}

\begin{pro}\label{3p} The following are 5-periodic (different from the fixed point $P$) orbits for the operator (\ref{w1}):
$$ e_1\to e_2\to e_3\to e_4\to e_5\to e_1.$$
$$\left({1\over 3}, {1\over 3}, {1\over 3}, 0, 0\right)\to \left(0, {1\over 3}, {1\over 3}, {1\over 3}, 0\right)\to \left(0, 0, {1\over 3}, {1\over 3}, {1\over 3}\right)\to \left({1\over 3}, 0, 0, {1\over 3}, {1\over 3}\right) $$
$$\to \left({1\over 3},  {1\over 3}, 0, 0, {1\over 3}\right)\to \left( {1\over 3}, {1\over 3},  {1\over 3}, 0, 0\right).$$
\end{pro}
\begin{proof} Note that assuming $x_1x_2x_3x_4x_5>0$ one gets from (\ref{pe2}) that $x_1=x_2=x_3=x_4=x_5=1/5$, i.e. the fixed point $P$.
 Now assume $x_1x_2x_3x_4x_5=0$, then considering all possibilities of $x_i=0$
 it is easy to see that all solutions of the system (\ref{pe2}) are mentioned in the proposition.
\end{proof}
Denote by $0(x)$ the number of zero coordinates of $x\in S^4$.
The following lemma follows from the definition of the operator (\ref{w1}).
\begin{lemma} $0(x)=0(W(x))$, for any $x\in S^4$.
\end{lemma}
For subset $I\subset \{1,2,3,4,5\}$ denote
$$\Gamma_I=\{x\in S^4 : x_i=0 \ \mbox{for all} \ \ i\in I\}.$$
Note that if number of elements $|I|$ in $I$ is equal to three then $\Gamma_I$ is a one-dimensional
edge connecting two vertices of the simplex.

\subsection{The trajectories on the one-dimensional boundary}
In this subsection consider initial points $x^{(0)}\in \partial S^4$ such that $0(x^{(0)})=3$.
There are 10 possibilities to put 3 zeros in a 5-dimensional vector. Operator (\ref{w1})
permutes zeros according to permutation $T_\pi$. It is easy to see that 10 possibilities
will be separated to two classes (each containing 5 vectors). These two classes considered in the following two cases:

{\bf Case 1:} Let $x^{(0)}=(x_1^{(0)}, 1- x_1^{(0)}, 0, 0, 0)$, we are going to study
trajectory of this point. Under action of operator (\ref{w1}) we get the following sequences
$$\left(x_1^{(0)}, 1- x_1^{(0)}, 0, 0, 0\right)\to \left(0, \underbrace{x_1^{(0)}(1+p(1-x_1^{(0)}))}_{x_2^{(1)}}, 1- x_2^{(1)}, 0, 0\right)$$
$$\to \left(0, 0, \underbrace{x_2^{(1)}(1+p(1-x_2^{(1)}))}_{x_3^{(2)}}, 1- x_3^{(2)}, 0\right)
\to \left(0, 0, 0, \underbrace{x_3^{(2)}(1+p(1-x_3^{(2)}))}_{x_4^{(3)}}, 1- x_4^{(3)}\right)$$
$$\to \left(\underbrace{(1- x_4^{(3)})(1-px_4^{(3)})}_{x_1^{(4)}}, 0, 0, 0, 1-x_1^{(4)}\right)
\to \left(\underbrace{(1- x_1^{(4)})(1+px_1^{(4)})}_{x_1^{(5)}}, 1-x_1^{(5)}, 0, 0, 0\right).$$
From this sequence of vectors we see that $x^{(0)}$ and $x^{(5)}$ have the same form, i.e., on the same edge. If we continue
the iteration of $W$ then we get similar travels, therefore, from the above sequence it follows that between $x_1^{(0)}$ and  $x_1^{(5)}$
we have the following relation:
$$x_1^{(5)}=F(x_1^{(0)}),$$
where $F(x)=h(g(f(f(f(x)))))$ with
\begin{equation}\label{fgh}
f(x)=x(1+p(1-x)), \ \ g(x)=(1-x)(1-px), \ \ h(x)=(1-x)(1+px).
\end{equation}
Thus we showed that $x_1^{(5k)}=F^k(x_1^{(0)})$. Now we study the dynamical system of $F$ on $[0,1]$.

\begin{lemma} The function $F$ has only two fixed points 0 and 1.
\end{lemma}
\begin{proof} It is easy to see that $F(0)=0$ and $F(1)=1$. We show that
there is no any other fixed point. For this it suffices to show that $F$ is monotone increasing, and
for $p>0$ it is concave and for $p<0$ it is convex.
For any $p\in [-1,1]$ we have
$$f'(x)=1+p-2px>0, \ \ g'(x)=-f'(x)<0, \ \ h'(x)=-(1-p+2px)<0.$$
Using these inequalities we get $F'(x)>0$, i.e., $F$ is increasing.
Denoting $\xi(x)=h(g(x))$ we get
$$\xi''(x)=\left\{\begin{array}{ll}
>0, \ \ \mbox{if} \ \ p<0\\[2mm]
<0, \ \ \mbox{if} \ \ p>0
\end{array}\right.$$
Moreover
$$f''(x)=\left\{\begin{array}{ll}
>0, \ \ \mbox{if} \ \ p<0\\[2mm]
<0, \ \ \mbox{if} \ \ p>0
\end{array}\right.$$
Combining these inequalities we get
$$F''(x)=\left\{\begin{array}{ll}
>0, \ \ \mbox{if} \ \ p<0\\[2mm]
<0, \ \ \mbox{if} \ \ p>0
\end{array}\right.$$
Therefore, there is no fixed point in $(0,1)$.
\end{proof}
\begin{pro}\label{ph} For any  $x^{(0)}=(x_1^{(0)}, 1- x_1^{(0)}, 0, 0, 0)$, we have
\begin{itemize}
\item If $p>0$ then
$$\lim_{n\to\infty}W^n(x^{(0)})=\left\{\begin{array}{lllll}
e_1, \ \ \mbox{if} \ \ n=5k\\[2mm]
e_2, \ \ \mbox{if} \ \ n=5k+1\\[2mm]
e_3, \ \ \mbox{if} \ \ n=5k+2\\[2mm]
e_4, \ \ \mbox{if} \ \ n=5k+3\\[2mm]
e_5, \ \ \mbox{if} \ \ n=5k+4
\end{array}\right.$$
\item If $p<0$ then
$$\lim_{n\to\infty}W^n(x^{(0)})=\left\{\begin{array}{lllll}
e_2, \ \ \mbox{if} \ \ n=5k\\[2mm]
e_3, \ \ \mbox{if} \ \ n=5k+1\\[2mm]
e_4, \ \ \mbox{if} \ \ n=5k+2\\[2mm]
e_5, \ \ \mbox{if} \ \ n=5k+3\\[2mm]
e_1, \ \ \mbox{if} \ \ n=5k+4
\end{array}\right.$$
\end{itemize}
\end{pro}
\begin{proof} The limits follow from the limit of $x_1^{(5k)}$ which by concavity (resp. convexity) of $F$ and the type of fixed points
have 1 (resp. 0) as a limit point, because $F'(0)=(1+p)^5$ and $F'(1)=(1-p)^5$.
\end{proof}

{\bf Case 2:} Let $x^{(0)}=(x_1^{(0)}, 0, 1- x_1^{(0)}, 0, 0)$,then under action of operator (\ref{w1})
we get the following sequences
$$\left(x_1^{(0)}, 0, 1- x_1^{(0)}, 0, 0\right)\to \left(0, \underbrace{x_1^{(0)}(1-p(1-x_1^{(0)}))}_{x_2^{(1)}}, 0, 1- x_2^{(1)}, 0\right)$$
$$\to \left(0, 0, \underbrace{x_2^{(1)}(1-p(1-x_2^{(1)}))}_{x_3^{(2)}}, 0, 1- x_3^{(2)}\right)
\to \left(\underbrace{(1-x_3^{(2)})(1+px_3^{(2)})}_{x_1^{(3)}}, 0, 0, 1- x_1^{(3)}, 0\right)$$
$$\to \left(0, \underbrace{x_1^{(3)}(1+p(1-x_1^{(3)}))}_{x_2^{(4)}}, 0, 0, 1-x_2^{(4)}\right)
\to \left(\underbrace{(1- x_2^{(4)})(1-px_2^{(4)})}_{x_1^{(5)}}, 0, 1-x_1^{(5)}, 0, 0\right).$$
Thus $x^{(0)}$ and $x^{(5)}$ are on the same edge, consequently between $x_1^{(0)}$ and  $x_1^{(5)}$
we have the following relation:
$$x_1^{(5)}=G(x_1^{(0)}),$$
where $G(x)=g(f(h(\alpha(\alpha(x)))))$ with $\alpha(x)=x(1-p(1-x))$ and other functions defined in (\ref{fgh}).
We note that $G$ on $[0,1]$ has similar properties as $F$ mentioned in Case 1.
In this case we have $G'(0)=(1-p)^5$ and $G'(1)=(1+p)^5$.
Thus we get
\begin{pro} For any  $x^{(0)}=(x_1^{(0)}, 0, 1- x_1^{(0)}, 0, 0)$, we have
\begin{itemize}
\item If $p>0$ then
$$\lim_{n\to\infty}W^n(x^{(0)})=\left\{\begin{array}{lllll}
e_3, \ \ \mbox{if} \ \ n=5k\\[2mm]
e_4, \ \ \mbox{if} \ \ n=5k+1\\[2mm]
e_5, \ \ \mbox{if} \ \ n=5k+2\\[2mm]
e_1, \ \ \mbox{if} \ \ n=5k+3\\[2mm]
e_2, \ \ \mbox{if} \ \ n=5k+4
\end{array}\right.$$
\item If $p<0$ then
$$\lim_{n\to\infty}W^n(x^{(0)})=\left\{\begin{array}{lllll}
e_1, \ \ \mbox{if} \ \ n=5k\\[2mm]
e_2, \ \ \mbox{if} \ \ n=5k+1\\[2mm]
e_3, \ \ \mbox{if} \ \ n=5k+2\\[2mm]
e_4, \ \ \mbox{if} \ \ n=5k+3\\[2mm]
e_5, \ \ \mbox{if} \ \ n=5k+4
\end{array}\right.$$
\end{itemize}
\end{pro}
As a corollary of these propositions we have
\begin{thm}\label{tt1} For any initial point $x^{(0)}$ in one-dimensional boundary of $S^4$ (i.e. $0(x^{(0)})=3$)
the set of limit points of its trajectory (under operator (\ref{w1})) is $\omega(x^{(0)})=\{e_1,e_2,e_3,e_4,e_5\}$.
\end{thm}
\subsection{The trajectories on the two-dimensional boundary}
Consider the case $0(x^{(0)})=2$. This case also can be divided to two classes:

{\it Case 1}. Let $x^{(0)}=(x_1^{(0)}, x_2^{(0)}, 1- x_1^{(0)}-x_2^{(0)}, 0, 0)$, then under action of operator (\ref{w1})
we get the following sequences
$$\left(x_1^{(0)}, x_2^{(0)}, 1- x_1^{(0)}-x_2^{(0)}, 0, 0)\right)$$
$$\to \left(0, \underbrace{x_1^{(0)}[1-p(1-x_1^{(0)}-2x_2^{(0)})]}_{x_2^{(1)}}, \, \underbrace{x_2^{(0)}[1+p(1-2x_1^{(0)}-x_2^{(0)})]}_{x_3^{(1)}}, \,
1- x_2^{(1)}-x_3^{(1)}, 0\right)$$
$$\to \left(0, 0, \underbrace{x_2^{(1)}[1-p(1-x_2^{(1)}-2x_3^{(1)})]}_{x_3^{(2)}}, \, \underbrace{x_3^{(1)}[1+p(1-2x_2^{(1)}-x_3^{(1)})]}_{x_4^{(2)}}, \,
1- x_3^{(2)}-x_4^{(2)}\right)$$
$$\to \left(\underbrace{(1-x_3^{(2)}-x_4^{(2)})[1+p(x_3^{(2)}-x_4^{(2)})]}_{x_1^{(3)}}, 0, 0,
\underbrace{x_3^{(2)}[1-p(1-x_3^{(2)}-2x_4^{(2)})]}_{x_4^{(3)}}, \, 1- x_1^{(3)}-x_4^{(3)}\right)$$
$$\to \left(\underbrace{(1-x_1^{(3)}-x_4^{(3)})[1+p(x_1^{(3)}-x_4^{(3)})]}_{x_1^{(4)}}, \,
 \underbrace{x_1^{(3)}[1-p(1-x_1^{(3)}-2x_4^{(3)})]}_{x_2^{(4)}}, 0, 0, \, 1- x_1^{(4)}-x_2^{(4)}\right)$$
$$\to \left(\underbrace{(1-x_1^{(4)}-x_2^{(4)})[1+p(x_1^{(4)}-x_2^{(4)})]}_{x_1^{(5)}}, \,
 \underbrace{x_1^{(4)}[1-p(1-x_1^{(4)}-2x_2^{(4)})]}_{x_2^{(5)}},  \, 1- x_1^{(5)}-x_2^{(5)}, 0, 0,\right)$$

Thus between {\it non-zero} coordinates (i.e. first three coordinates) of the vectors $x^{(0)}$ and  $x^{(5)}$
we have the following relation:
\begin{equation}\label{AB}
\hat x^{(5)}=B(B(B(A(A(\hat x^{(0)}))))),
\end{equation}
where $\hat x=(x_1,x_2,x_3)$, and
$A, B: S^2\to S^2$ are defined as the following
$$A: \left\{\begin{array}{lll}
x'=x[1+p(y-z)]\\[2mm]
y'=y[1+p(z-x)]\\[2mm]
z'=z[1+p(x-y)]
\end{array}\right.$$
$$B: \left\{\begin{array}{lll}
x'=z[1+p(x-y)]\\[2mm]
y'=x[1+p(y-z)]\\[2mm]
z'=y[1+p(z-x)]
\end{array}\right.$$
Note that $A$ is a Volterra operator and $B$ is its permuted one.
Consider permutation $\tau=\left(\begin{array}{ll}
1 \ \ 2 \ \ 3\\[2mm]
3 \ \ 1 \ \ 2
\end{array}\right)$. It is easy to check that
$$B=A(T_\tau)=T_\tau(A),$$
$$B^{(n)}=\left\{\begin{array}{lll}
A^{(n)}, \ \ \mbox{if} \ \ n=3k, \\[2mm]
T_\tau(A^{(n)}), \ \ \mbox{if} \ \ n=3k+1\\[2mm]
T_{\tau^2}(A^{(n)}), \ \ \mbox{if} \ \ n=3k+2,
\end{array}\right.$$
where $\tau^2=\tau^{-1}=\left(\begin{array}{ll}
1 \ \ 2 \ \ 3\\[2mm]
2 \ \ 3 \ \ 1
\end{array}\right)$.
Therefore from (\ref{AB}) we get
$$\hat x^{(5)}=A^{(5)}(\hat x^{(0)}).$$
The following result is known for the operator $A$ (see \cite{G}):
\begin{itemize}
\item The fixed points are $\hat e_1=(1, 0, 0), \hat e_2=(0, 1, 0), \hat e_3=(0, 0, 1), \hat P=(1/3, 1/3, 1/3)$.
The points $\hat e_i$ are saddle, the point $\hat P$ is repeller independently on $p\ne 0$.

\item For any $x\in {\rm int} S^2$, $x\ne \hat P$, the set $\omega_A(x)$ of limit points of the trajectory
of the point $x$ is infinite (i.e. the limit does not exist) and lies on the boundary of the simplex.
\end{itemize}
Therefore $\hat x^{(5k)}$ has no a limit point, but only one of its coordinate goes to 0. Indeed if two coordinates
go to zero then remaining coordinate has limit 1, which is contradiction to non-existence of the limit.
\begin{thm}\label{tt2} For any initial point $x^{(0)}=(x_1^{(0)}, x_2^{(0)}, 1- x_1^{(0)}-x_2^{(0)}, 0, 0)$, with $\hat x^{(0)}\ne \hat P$ in two-dimensional
boundary of $S^4$ the set of limit points of its trajectory (under operator (\ref{w1})) is infinite and lies
on the one-dimensional boundary.
\end{thm}
\begin{proof} Follows from the above mentioned properties of $\hat x^{(5k)}$.
\end{proof}
{\it Case 2}. Let $x^{(0)}=(x_1^{(0)}, x_2^{(0)}, 0, 1- x_1^{(0)}-x_2^{(0)}, 0)$, then under action of operator (\ref{w1}),
similarly to the previous case we get (see Lemma \ref{lT}) that $\hat x^{(5k)}=\hat V^{(5k)} (\hat x^{(0)})$,
where $\hat V$ is the following Volterra
operator
$$\hat V: \left\{\begin{array}{lll}
x'=x[1+p(y+z)]\\[2mm]
y'=y[1-p(x+z)]\\[2mm]
z'=z[1-p(x-y)]
\end{array}\right.$$
Note that this Volterra operator has only three fixed points $(1,0,0), (0,1,0), (0,0,1)$.
If $p>0$ (resp. $p<0$) then $x^{(n)}$ (resp. $y^{(n)}$) increases and  $y^{(n)}$ (resp. $x^{(n)}$) decreases.
Therefore these two coordinates have limit, consequently $z^{(n)}=1-x^{(n)}-y^{(n)}$ has also a limit.
The limit point is a fixed point of $\hat V$, and thus we get
$$\lim_{n\to\infty}\hat x^{(n)}=\left\{\begin{array}{ll}
(1,0,0), \ \ \mbox{if} \ \ p>0\\[2mm]
(0,1,0), \ \ \mbox{if} \ \ p<0
\end{array}\right.$$
Thus the following theorem is true
\begin{thm}\label{tt3} For any initial point $x^{(0)}=(x_1^{(0)}, x_2^{(0)}, 0, 1- x_1^{(0)}-x_2^{(0)}, 0)$,
with $\hat x^{(0)}\ne \hat P$ in two-dimensional
boundary of $S^4$ we have
\begin{itemize}
\item If $p>0$ then
$$\lim_{n\to\infty}W^n(x^{(0)})=\left\{\begin{array}{lllll}
e_1, \ \ \mbox{if} \ \ n=5k\\[2mm]
e_2, \ \ \mbox{if} \ \ n=5k+1\\[2mm]
e_3, \ \ \mbox{if} \ \ n=5k+2\\[2mm]
e_4, \ \ \mbox{if} \ \ n=5k+3\\[2mm]
e_5, \ \ \mbox{if} \ \ n=5k+4
\end{array}\right.$$
\item If $p<0$ then
$$\lim_{n\to\infty}W^n(x^{(0)})=\left\{\begin{array}{lllll}
e_2, \ \ \mbox{if} \ \ n=5k\\[2mm]
e_3, \ \ \mbox{if} \ \ n=5k+1\\[2mm]
e_4, \ \ \mbox{if} \ \ n=5k+2\\[2mm]
e_5, \ \ \mbox{if} \ \ n=5k+3\\[2mm]
e_1, \ \ \mbox{if} \ \ n=5k+4
\end{array}\right.$$
\end{itemize}
\end{thm}
\subsection{Three dimensional boundary}
Take an initial point of the form $x^{(0)}=(x^{(0)}_1, x^{(0)}_2, x^{(0)}_3, 1-x^{(0)}_1-x^{(0)}_2-x^{(0)}_3, 0)$. It is clear that
under action of $W$, zero will 'travel' all coordinates of the trajectory vectors.
  By Lemma \ref{lT} we get $x^{(5k)}=\tilde V^{(5k)}(x^{(0)})$, where $\tilde V: S^3\to S^3$ is defined as
$$\tilde V: \left\{\begin{array}{llll}
x'=x[1+p(y-z+t)]\\[2mm]
y'=y[1-p(x-z+t)]\\[2mm]
z'=z[1+p(x-y+t)]\\[2mm]
t'=t[1-p(x-y+z)]
\end{array}\right.$$

It is clear that any face of the simplex $S^4$ is invariant with respect to $\tilde V$.
But here we will not consider the faces because in this subsection we assumed that boundary point has four non-zero coordinates.

\begin{pro}\label{pp} The operator $\tilde V$ has the following properties:
\begin{itemize}
\item[1.] The fixed points are
 $$\tilde e_i=(\delta_{1i}, \delta_{2i}, \delta_{3i}, \delta_{4i}), i=1,2,3,4; \ \ \
M=({1\over 3}, {1\over 3}, {1\over 3}, 0), N=(0, {1\over 3}, {1\over 3}, {1\over 3}).$$
\item[2.] The interval $MN=\{({1-u\over 3}, {1\over 3}, {1\over 3}, {u\over 3}): u\in [0,1]\}$ is invariant with respect to $\tilde V$.
\item[3.] If $\tilde x^{(0)}\in MN$ then
$$\lim_{n\to \infty}\tilde V^n(\tilde x^0)=\left\{\begin{array}{ll}
M, \ \ \mbox{if} \ \ p>0\\[2mm]
N, \ \ \mbox{if} \ \ p<0
\end{array}\right.$$
\item[4.] If $\tilde x^{(0)}\notin MN$ (with four non-zero coordinates)
then $\tilde V^n(\tilde x^0)$ does not converge and its
limit points set is infinite and lies on the boundary of $S^3$.
\end{itemize}
\end{pro}
\begin{proof} 1. Follows from simple analysis of the equation $\tilde V(x)=x$.

2. We take arbitrary element $\eta=(u, {1\over 3}, {1\over 3}, {1\over 3}-u)\in MN$, (where $u\in [0,1/3]$) and show that $\tilde V(\eta)\in MN$:
we have
$$\tilde V(\eta): \left\{\begin{array}{llll}
x'=u[1+p(1/3-u)]\\[2mm]
y'=(1/3)[1-p(u-(1/3)+((1/3)-u))]=1/3\\[2mm]
z'=(1/3)[1+p(u-(1/3)+((1/3)-u))]=1/3\\[2mm]
t'=((1/3)-u)[1-p(u-(1/3)+(1/3))]=1/3-x'
\end{array}\right.$$
Hence $\tilde V(\eta)\in MN$.

3. Restriction of the operator $\tilde V$ on $MN$ gives the one dimensional mapping $\psi: [0, {1\over 3}]\to [0, {1\over 3}]$,
with $\psi(u)=u(1+p({1\over 3}- u))$. This function is monotone increasing on $[0, {1\over 3}]$,  concave for $p>0$ and convex for $p<0$.
It has two fixed points $0$ and $1/3$ with
$\psi'(0)=1+p/3$ and $\psi'(1/3)=1-p/3$. From these properties follows the part 3.

4. This property is known for general Volterra operators (see \cite{G} and Theorem 2.4 in \cite{GMR}).
\end{proof}
Denote
$$\overline{MN}=\{(x,0): x\in MN\}.$$
For a set $U\subset S^4$ denote
$$T_\pi U=\{T_\pi u: \ \ u\in U\}.$$
$$U^\pi=U\cup T_\pi U\cup T^2_\pi U\cup T^3_\pi U\cup T^4_\pi U.$$

For any $x\in \partial S^4$ denote by $\tilde x$ the vector constructed
from $x$ by omitting zero coordinates.

As a corollary of Proposition \ref{pp} we get
\begin{thm}\label{tt4} For any point $x^{(0)}=(x^{(0)}_1, x^{(0)}_2, x^{(0)}_3, 1-x^{(0)}_1-x^{(0)}_2-x^{(0)}_3, 0)$ we have
\begin{itemize}
\item[a)] If $ x^{(0)}\in \overline{MN}^\pi$ then the trajectory $x^{(n)}$ converges to 5-periodic orbit started at $\left({1\over 3}, {1\over 3}, {1\over 3}, 0, 0\right)$ mentioned in Proposition \ref{3p}.
\item[b)] If  $x^{(0)}\notin \overline{MN}^\pi$   the trajectory does not converge and its
limit points set is infinite and lies on the boundary of $S^4$.
\end{itemize}
\end{thm}

\subsection{Inside of the simplex}

\begin{thm}\label{tt5} For any $x^{(0)}\in {\rm int}S^4$ the set of limit point of its trajectory (under operator (\ref{w1}), denoted by $W$)
is infinite and lies on the boundary of the simplex $S^4$.
\end{thm}
\begin{proof} For any $x\in {\rm int} S^4$ denote $\varphi(x)=x_1x_2x_3x_4x_5$.
Estimate $\varphi(x')=\varphi(W(x))$, using (\ref{ch1}) we get
$$\varphi(x')\leq \varphi(x).$$
Iterating this inequality we get  $\varphi(x^{(n+1)})\leq \varphi(x^{(n)}).$
Thus the limit  $\lim_{n\to \infty}\varphi(x^{(n)})=\phi\geq 0$ exists.
We claim that $\phi=0$. Indeed, assuming on the contrary $\phi>0$ we get
\begin{equation}\label{bir}
1=\lim_{n\to \infty} {\varphi(x^{(n)})\over \varphi(x^{(n-1)})}=\lim_{n\to \infty}\psi(x^{(n)}).
\end{equation}
As in (\ref{ch1}) and (\ref{ch2}) we have
\begin{equation}\label{ps}
\max_{x\in S^4}\psi(x)=\psi(x^*)=1 \ \ \mbox{if and only if} \ \
x^*=P=\left({1\over 5}, {1\over 5},  {1\over 5}, {1\over 5}, {1\over 5}\right).
\end{equation}
Next we show that if $\phi>0$, then $x^{(n)}\to P$ as $n\to \infty$.
Suppose the converse: there is a sequence $\{n_k\}_{k=1,2,\dots}$ such that
\begin{equation}\label{ne}
\lim_{k\to\infty}x^{(n_k)}=\nu\ne P.
\end{equation}
Since $\psi$ is a continuous function, it follows from (\ref{ps}) and (\ref{ne}) that
\begin{equation}\label{nee}
\lim_{k\to\infty}\psi(x^{(n_k)})=\psi(\nu)<1.
\end{equation}
But since $\nu\ne P$, inequality (\ref{nee}) contradicts (\ref{bir}) and hence $x^{(n)}\to P$ as $n\to \infty$.
But this is impossible, because by Proposition \ref{pb} the unique fixed point $P$ is repeller. Thus
\begin{equation}\label{0b}
 \lim_{n\to \infty}\varphi(x^{(n)})=\lim_{n\to \infty}x^{(n)}_1x^{(n)}_2x^{(n)}_3x^{(n)}_4x^{(n)}_5=0.
 \end{equation}
 Since $\varphi(x)>0$ for all $x\in {\rm int}S^4$, it follows that the equality $\varphi(x)=0$ can only be
 on $\partial S^4$, consequently (\ref{0b}) gives $\omega(x^{(0)})\subset \partial S^4$.

To show that  $\omega(x^{(0)})$ is an infinite set, we use Lemma \ref{lT},
i.e., consider the Volterra operator $V$, defined by $V=T^{-1}_\pi\circ W$ (see (\ref{w1})).

For general Volterra operator the following is known (see Theorem 2.4 in \cite{GMR}):

{\it  If a Volterra operator has an isolated fixed point $x^*\in {\rm int}S^{m-1}$, then for any initial
point $x^{(0)} \in {\rm int}S^{m-1}\setminus \{x^*\}$ the trajectory $\{x^{(n)}\}$ does not converge}.

This fact and Lemma \ref{lT} completes the proof.
\end{proof}

\section{Conclusion}
In this paper we presented quadratic operator describing
the dynamics of the CFEP. Let $x = (x_1, x_2, x_3, x_4, x_5)\in S^4$ be an
initial state, i.e. the probability distribution on the set
$$\{Wood, Fire,  Earth,  Metal,  Water\}.$$

The following are interpretations of our results to CFEP:

\begin{itemize}
\item Non existence of a fixed point on the boundary (Lemma \ref{lf1}) means that if some element absent at initial time (i.e.
$x_i=0$ for some $i$) then in future this element surely appears.

\item The repeller unique fixed point (Proposition \ref{pb}) means that the CFEP system
does not have an equilibrium state (if at initial time it was out of the equilibrium).

\item Existence of 5-periodic orbits $e_1\to \dots \to e_5\to e_1$ (Proposition \ref{3p}) means that if at initial time
we had only one element then periodically (by the direction of generating interactions, see Fig. \ref{ff1})
at each future time we will have only one of five elements.

\item Existence of 5-periodic orbits consisting ${1\over 3}$ (Proposition \ref{3p}) means that if at initial time
with equal probabilities we had only three consecutive elements (by the direction of generating interactions)
then at each future time we will have only three consecutive elements.

\item A convergence to the periodic orbit $e_1\to \dots \to e_5\to e_1$ (Theorem \ref{tt1})
means that if at an initial time we had only two elements (not necessary consecutive ones)
then as time goes to infinity the state of the CFEP will
change periodically, and at each time only one of elements will be present.

 \item Theorem \ref{tt2} says that if at an initial time we had only three consecutive elements
then as time goes to infinity the states of the CFEP do not have any periodicity,
 but asymptotically at each time only two elements remain.

\item Theorem \ref{tt3} has interpretation that if at an initial time we had only three non-consecutive elements
then as time goes to infinity the state of the CFEP will
change periodically, and at each time only one of elements will be present.

\item By Theorem \ref{tt4} we conclude that if initially we had four elements (two of them with probability ${1\over 3}$)
then future states of CFEP will be asymptotically 5-periodic containing only three elements. If initially we had four
elements (but not with two of them with probability ${1\over 3}$) then as time goes to infinity the states of the CFEP do not have any periodicity,
 but asymptotically at each time only three elements remain.

 \item If initially we had five
elements (Theorem \ref{tt5}) then as time goes to infinity the states of the CFEP do not have any periodicity,
 but asymptotically at each time four elements remain.

   \end{itemize}

\section*{ Acknowledgements}

 Rozikov's work was supported as a visiting scholar at International Islamic
 University Malaysia through the MOHE Grant (FRGS17-027-0593).

\end{document}